\documentclass[11pt]{amsart}
\usepackage[T1]{fontenc}
\usepackage{lmodern}
\usepackage[margin=1in]{geometry}
\usepackage{amsmath,amssymb,amsthm,mathtools,booktabs,longtable,array}
\usepackage{microtype,needspace,xurl,xcolor}
\usepackage[font=small,labelfont=sc,labelsep=period]{caption}
\usepackage{hyperref}
\hypersetup{colorlinks=true,linkcolor=blue!45!black,citecolor=blue!45!black,urlcolor=blue!45!black,pdftitle={CM fibers of Shimura parametrizations and thirteen proofs of the class number one theorem},pdfauthor={Jingxuan Geng}}
\newtheorem{theorem}{Theorem}[section]
\newtheorem{proposition}[theorem]{Proposition}
\newtheorem{corollary}[theorem]{Corollary}
\newtheorem{conjecture}[theorem]{Conjecture}
\newtheorem{lemma}[theorem]{Lemma}
\theoremstyle{definition}
\newtheorem{definition}[theorem]{Definition}
\newtheorem{assumption}[theorem]{Assumption}
\newcommand{\Q}{\mathbf Q}
\newcommand{\Z}{\mathbf Z}
\newcommand{\R}{\mathbf R}

\newcommand{\C}{\mathbf C}
\newcommand{\F}{\mathbf F}
\newcommand{\OO}{\mathcal O}
\newcommand{\CM}{\operatorname{CM}}
\newcommand{\Gal}{\operatorname{Gal}}
\newcommand{\Fix}{\operatorname{Fix}}
\newcommand{\Pic}{\operatorname{Pic}}
\newcommand{\tors}{\mathrm{tors}}
\newcommand{\labelE}[1]{\texttt{#1}}
\DeclareMathOperator{\GL}{GL}
\numberwithin{equation}{section}
\title[CM fibers and class number one]{CM fibers of Shimura parametrizations\newline and thirteen proofs of the class number one theorem}
\author{Jingxuan Geng}
\date{October 5, 2026}
\begin{document}
\begin{abstract}
We determine thirteen entirely CM identity fibers of Shimura parametrizations attached to elliptic curves of rank two and conductor below $2000$. Each yields a proof of the imaginary quadratic class number one theorem. The argument combines the Gross--Zagier formula with the Ribet--Takahashi degree comparison to show that the fiber is completely explained by known CM points.
\end{abstract}
\maketitle
\tableofcontents

\section{Introduction}\label{sec:introduction}

Gauss's class number one problem asks which imaginary quadratic fields have unique factorization in their rings of integers. The answer, in terms of fundamental discriminants, is
\[
 \{-3,-4,-7,-8,-11,-19,-43,-67,-163\}.
\]
Heegner's modular-function argument appeared in 1952~\cite{Heegner1952}. Baker proved the classification using linear forms in logarithms~\cite{Baker1966,Baker1969}, and Stark gave an independent proof and subsequently justified the disputed step in Heegner's argument~\cite{Stark1967,Stark1969}.

Serre explained the geometric principle underlying modular-curve proofs: complex multiplication associates a class number one field with an integral point on a suitable modular curve~\cite[Appendix]{Serre1997}. Determining those points excludes further fields. Later examples include Schoof--Tzanakis at level~11~\cite{SchoofTzanakis2012} and the quadratic Chabauty computation on the nonsplit Cartan curve of level~17 by Balakrishnan, Dogra, M\"uller, Tuitman, and Vonk~\cite{BDMTV2023}.

Goldfeld connected the problem with central vanishing of elliptic-curve $L$-functions~\cite{Goldfeld1976}. The Gross--Zagier formula relates a central derivative to a Heegner-point height~\cite{GrossZagier1986}; together with a suitable analytic-rank-three example, it yields Goldfeld's effective bound $h(D)\gg_{\varepsilon}(\log |D|)^{1-\varepsilon}$ for every $\varepsilon>0$, making class-number problems effectively solvable~\cite{BuhlerGrossZagier1985,Oesterle1985}. More recently, OpenAI's uniform exclusion of Landau--Siegel zeros~\cite{OpenAI2026}, with explicit constants supplied by Banks~\cite{Banks2026}, yields via the analytic class number formula the effective bound $h(D)\gg\sqrt{|D|}/\log |D|$ for negative fundamental discriminants $D$.

In this paper, we give a new geometric approach to class number one theorem by studying the identity fiber of a normalized Shimura parametrization $\varphi_E:X^{pq}\to E$ for an elliptic curve of rank at least two.  Gross--Zagier forces every admissible class-number-one CM point into this fiber. If known CM points exhaust it, Oesterl\'e's small-prime estimate bounds every possible class number one discriminant, thus finishing the class number one problem. We prove exhaustion by matching the CM contribution, with ramification, to the parametrization degree. 

Our approach is close in spirit to the work of Caeiro and Darmon~\cite{CaeiroDarmon2025,CaeiroDarmonUnpublished}. They conditionally solve class-number-one problems in certain real quadratic families by forcing Stark--Heegner parameters to be zeros of explicit rigid analytic functions and analyzing those zeros. In joint work with Caeiro~\cite{CaeiroGengUnpublished}, we use $p$-adic uniformization to study the fiber of Shimura parametrization and its higher analogue, giving another approach to class number problems.

\begin{samepage}
\begin{theorem}\label{thm:fibers}
For each of the elliptic curves with LMFDB labels
\begin{gather*}
\labelE{446.a1},\ \labelE{681.a1},\ \labelE{707.a1},\ \labelE{718.a1},\ \labelE{817.a1},\
\labelE{1077.a1},\ \labelE{1126.a1},\ \labelE{1137.a1},\\
\labelE{1466.b1},\ \labelE{1477.a1},\ \labelE{1594.a1},\ \labelE{1642.a1},\ \labelE{1909.a1},
\end{gather*}
there is a parametrization $\varphi_E:X^{N_E}\to E$ over $\Q$ whose identity fiber is entirely accounted for by known CM points, counted with their ramification multiplicities, as recorded in the following table.
\end{theorem}
\end{samepage}
\begingroup
\small
\setlength{\tabcolsep}{3pt}
\renewcommand{\arraystretch}{1.12}

\begin{longtable}{@{}l
  >{\raggedright\arraybackslash}p{0.30\textwidth}
  >{\raggedright\arraybackslash}p{0.22\textwidth}
  >{\raggedright\arraybackslash}p{0.25\textwidth}
  r@{}}

\caption{Complete CM fibers of the above elliptic curves. 
An entry $(D,n)$ denotes the $n$ points of discriminant $D$ on $X^{N_E}$.
A star indicates ramification index two.}
\label{tab:complete-fibers}\\
\toprule
$E$
& Class number one points
& Atkin--Lehner fixed points\newline ($h(D)>1$)
& Additional CM points\newline ($h(D)>1$)
& $\delta_E$\\
\midrule
\endfirsthead

\toprule
$E$
& Class number one points
& Atkin--Lehner fixed points\newline ($h(D)>1$)
& Additional CM points\newline ($h(D)>1$)
& $\delta_E$\\
\midrule
\endhead

\bottomrule
\endfoot

\texttt{446.a1}
& $(-4,2)$, $(-8,2)$, $(-19,4)$, $(-43,4)$
& $(-1784,32)$
& ---
& 44\\[4pt]

\texttt{681.a1}
& $(-3,2)$, $(-4,4)$, $(-7,4)$, $(-12,2)$,
  $(-16,4)$, $(-19,4)$, $(-28,4)$, $(-43,4)$
& $(-2724,20)$
& ---
& 48\\[4pt]

\texttt{707.a1}
& $(-7,2)$, $(-8,4)$, $(-11,4)$, $(-28,2)$,
  $(-67,4)$, $(-163,4)$
& $(-707,6)$, $(-2828,18)$
& $(-60,8)$
& 52\\[4pt]

\texttt{718.a1}
& $(-3,4)$, $(-4,2)$, $(-8,2)$, $(-11,4)$, $(-27,4)$
& $(-2872,12)$
& ---
& 28\\[4pt]

\texttt{817.a1}
& $(-4,4)$, $(-11,4)$, $(-16,4)$, $(-43,2)^{*}$
& $(-3268,12)$
& ---
& 28\\[4pt]

\texttt{1077.a1}
& $(-3,2)$, $(-4,4)$, $(-12,2)$, $(-16,4)$
& $(-4308,24)$
& ---
& 36\\[4pt]

\texttt{1126.a1}
& $(-3,4)$, $(-4,2)$, $(-11,4)$, $(-19,4)$,
  $(-27,4)$, $(-67,4)$
& $(-4504,22)$
& ---
& 44\\[4pt]

\texttt{1137.a1}
& $(-4,4)$, $(-16,4)$, $(-19,4)$, $(-67,4)$
& $(-4548,20)$
& ---
& 36\\[4pt]

\texttt{1466.b1}
& $(-8,2)^{*}$, $(-11,4)$, $(-43,4)$, $(-67,4)$
& $(-2932,14)$, $(-5864,58)$
& $(-20,4)$, $(-24,4)$, $(-52,4)$, $(-72,4)$
& 104\\[4pt]

\texttt{1477.a1}
& $(-4,4)$, $(-11,4)$, $(-16,4)$, $(-43,4)$, $(-163,4)$
& $(-5908,16)$
& $(-36,8)$, $(-148,8)$
& 52\\[4pt]

\texttt{1594.a1}
& $(-3,4)$, $(-8,2)^{*}$, $(-19,4)$, $(-27,4)$
& $(-3188,30)$, $(-6376,34)$
& $(-20,4)$, $(-72,4)$, $(-88,4)$, $(-148,4)$
& 96\\[4pt]

\texttt{1642.a1}
& $(-3,4)$, $(-8,2)$, $(-11,4)$, $(-27,4)$
& $(-6568,14)$
& ---
& 28\\[4pt]

\texttt{1909.a1}
& $(-3,4)$, $(-4,4)$, $(-12,4)$, $(-16,4)$, $(-27,4)$
& $(-7636,28)$
& $(-75,8)$, $(-100,8)$, $(-667,8)$
& 72\\

\end{longtable}
\endgroup

\begin{corollary}\label{cor:classone}
Each curve in Theorem~\ref{thm:fibers} gives a new proof of the imaginary quadratic class number one theorem.
\end{corollary}

Section~\ref{sec:class-background} recalls the genus-theoretic restriction on class number one discriminants and Oesterl\'e's inertness lemma. Section~\ref{sec:setup} recalls Shimura curves, CM points, and the Gross--Zagier formula, and then explains why Theorem~\ref{thm:fibers} implies Corollary~\ref{cor:classone}. Section~\ref{sec:involutions} discusses the degrees of Shimura uniformizations, Atkin--Lehner fixed points, and forced ramification. Section~\ref{sec:characters} accounts for additional CM points through the vanishing of twists and completes the proof of Theorem~\ref{thm:fibers}. Finally, Section~\ref{sec:further} proposes a finiteness conjecture.

\section{Class number one and nonsplit type}\label{sec:class-background}

For a negative fundamental discriminant $D$, write $K_D=\Q(\sqrt D)$ and $h(D)=\#\Pic(\OO_{K_D})$. We recall two classical restrictions in the class number one case.

\begin{lemma}[Genus theory]\label{lem:genus}
If $D<0$ is fundamental and $h(D)=1$, then $D=-4$, $D=-8$, or $D=-p$ for a prime $p\equiv3\pmod4$.
\end{lemma}

This is the immediate class-number-one specialization of the principal genus theorem: if $D$ has $t$ prime-discriminant factors, then $\#(\Pic(\OO_{K_D})/\Pic(\OO_{K_D})^2)=2^{t-1}$~\cite[Proposition~3.11 and Theorem~3.15]{Cox2013}.

\begin{definition}\label{def:types}
Let $N>1$ and let $D<0$ be fundamental. We say that $N$ is of \emph{nonsplit type modulo $D$} if every prime $\ell\mid N$ is inert or ramified in $K_D$, equivalently $(D/\ell)\in\{0,-1\}$. It is of \emph{inert type modulo $D$} if every such prime is inert, equivalently $(D/\ell)=-1$. Here $(D/\ell)$ is the Kronecker symbol.
\end{definition}

\begin{lemma}[Oesterl\'e's inertness lemma]\label{lem:inertness}
Let $D<0$ be fundamental with $h(D)=1$. If a prime $\ell$ satisfies $\ell<|D|/4$, then $\ell$ does not split in $K_D$. Thus, writing $P(N)=\max_{\ell\mid N}\ell$, the inequality $|D|>4P(N)$ implies that $N$ is of nonsplit type modulo $D$; it is of inert type if also $(D,N)=1$.
\end{lemma}

We use this standard small-prime estimate in the form recalled by Oesterl\'e~\cite[\S1.4(a)]{Oesterle1985}. For $N=pq$ with $p<q$, Lemma~\ref{lem:genus} shows that every hypothetical class number one discriminant satisfying $|D|>4q$ makes $N$ of inert type.

\section{Shimura parametrizations and the proof strategy}\label{sec:setup}

\subsection{The level-one Shimura curve}
Let \(B\) be an indefinite quaternion division algebra over \(\Q\) of reduced
discriminant \(N\), and fix a maximal order \(\OO_B\subset B\).  Thus \(N>1\)
is squarefree and is the product of an even number of primes.  We use only
the trivial-level curve attached to \(\OO_B\). 

Choose a splitting
\[
\iota_\infty:B\otimes_{\Q}\R\xrightarrow{\;\sim\;}M_2(\R)
\]
and put
\[
\Gamma^1
=
\iota_\infty(\OO_B^1)/\{\pm1\}
\subset\mathrm{PSL}_2(\R),
\qquad
\OO_B^1=\{\gamma\in\OO_B^\times:\operatorname{nrd}(\gamma)=1\}.
\]
The complex points of the level-one Shimura curve are
\[
X^N(\C)=\Gamma^1\backslash\mathfrak h,
\]
where \(\mathfrak h\) is the complex upper half-plane.  Since \(B\) is a
division algebra, this quotient is compact.  Shimura proved that this Riemann
surface has a canonical smooth projective model \(X^N/\Q\); see~\cite[\S\S43.1 and 43.8]{Voight2021}.

The level-one Atkin--Lehner group is
\[
W_N
=
N_{B^\times}(\OO_B)/\Q^\times\OO_B^\times
\simeq
\prod_{\ell\mid N}\Z/2\Z.
\]
We write \(W_\ell\) for its standard generator at \(\ell\mid N\).  These
involutions are defined over \(\Q\) and act on \(X^N\)
\cite[\S43.1]{Voight2021}.

\subsection{CM points}
Let $D=d_0f^2<0$, with $d_0$ fundamental, and put
$K_D=\Q(\sqrt{d_0})$, $\OO_D=\Z+f\OO_{K_D}$, and
$h(D)=\#\Pic(\OO_D)$. Denote by $\CM_N(D)$ the CM points of discriminant
$D$ on $X^N$, equivalently normalized optimal embeddings
$\OO_D\hookrightarrow\OO_B$ modulo conjugation by $\OO_B^1$
\cite[Definition~5.5]{GonzalezRotger2006};
\cite[\S~43.7]{Voight2021}.
We always assume $(f,N)=1$.

The set $\CM_N(D)$ is nonempty if and only if $N$ is of nonsplit type
modulo $d_0$~\cite[\S\S30.5--30.7]{Voight2021}. Every such point is defined
over the ring class field $H_D$ of $\OO_D$, and under the Artin isomorphism
$\Pic(\OO_D)\simeq\Gal(H_D/K_D)$, Galois acts through the ideal-class
action~\cite[\S43.8.2]{Voight2021}.

When $\CM_N(D)$ is nonempty, the commuting actions of $\Pic(\OO_D)$ and
\[
 W_N(D)=\langle W_\ell:\ell\mid N,\ \ell\nmid D\rangle
\]
give a simply transitive action of $\Pic(\OO_D)\times W_N(D)$ on
$\CM_N(D)$~\cite[Proposition~5.6]{GonzalezRotger2006}. In particular,
\[
 \#\CM_N(D)
 =h(D)\prod_{\ell\mid N}\left(1-\left(\frac{D}{\ell}\right)\right)
 =h(D)\,2^{\#\{\ell\mid N:\ell\nmid D\}}.
\]

\subsection{The Gross--Zagier--Zhang formula}\label{subsec:elliptic-gz}
For an elliptic curve $E/\Q$ of conductor $N$, modularity and Jacquet--Langlands give a nonzero quotient
$\alpha:\operatorname{Jac}(X^N)\to E$; when the rational isogeny class is a singleton, we take
the optimal quotient~\cite[\S\S2.6, 4.4--4.6]{Darmon2004}.
Let $\xi\in\Pic(X^N)\otimes\Q$ be the degree-one normalized Hodge class~\cite[p.~2528]{CaiShuTian2014}.
Suppose $D=df^2$, $K=\Q(\sqrt{D})$ and $N$ is of nonsplit type
modulo $d$ with $\gcd(f,N)=1$. For $\tau\in\CM_N(D)$, put
\[
 Z_{E,D}(\tau)=\sum_{\sigma\in\Gal(H_D/K)}
        \alpha([\tau^\sigma]-\xi)
        \quad\text{in }E(K)\otimes_\Z\Q.
\]

\begin{theorem}[Gross--Zagier--Zhang formula]\label{thm:untwisted-gz}
Suppose that
\[
 a_\ell=1\qquad\text{for every }\ell\mid\gcd(N,d).
\]
Then $L(E/K,s)=L(E,s)L(E^d,s)$ has functional-equation sign $-1$, and
\begin{equation}\label{eq:gz}
 \widehat h_E\bigl(Z_{E,D}(\tau)\bigr)
       =C_{E,D}L'(E/K,1),\qquad C_{E,D}\ge0.
\end{equation}
\end{theorem}

This version is the untwisted elliptic specialization of~\cite[Theorem~1.2]{YuanZhangZhang2013}, in the form given in~\cite[\S2C, equation~(2-4)]{CaiShuTian2014}.

\subsection{Proof strategy}
\label{subsec:normalization}

Analytic rank at least two gives
$L'(E/K,1)=0$, since $L(E/K,s)=L(E,s)L(E^d,s)$.
Theorem~\ref{thm:untwisted-gz} therefore gives $Z_{E,D}=0$ in $E(K)\otimes\Q$ for every admissible $D$.

We impose the following hypotheses on the elliptic curve under consideration.

\begin{assumption}\label{ass:working}
The elliptic curve $E/\Q$ satisfies the following conditions:
\begin{enumerate}
\renewcommand{\labelenumi}{(\roman{enumi})}
\item Its conductor is a product of two primes $N=pq$, with $p<q$.
\item Its analytic rank $r_{\operatorname{an}}(E)$ is at least two. By Gross--Zagier--Kolyvagin~\cite[Theorem~3.22]{Darmon2004}, it suffices to have algebraic
rank at least two.
\item Its rational isogeny class consists of a single curve.
\item There is a fundamental discriminant $d_0<0$ with $h(d_0)=1$ such that $N$ is of inert type modulo $d_0$.
\end{enumerate}
\end{assumption}
LMFDB records rank two and a singleton isogeny class for each of our thirteen curves~\cite{LMFDB}. 

\begin{lemma}\label{lem:biquadratic-torsion-free}
    For an elliptic curve $E/\Q$ satisfying Assumption \ref{ass:working}, $E(F)_{\tors}=0$ for any biquadratic field $F=\Q(\sqrt u,\sqrt v)$ .
\end{lemma}
\begin{proof}
    Conditions (i) and (iii) imply that $E[\ell]$ is irreducible for every prime $\ell$
and hence has Galois image $\GL_2(\F_\ell)$~\cite[Proposition~1.1]{RibetStein2001}.
If a biquadratic field $F=\Q(\sqrt u,\sqrt v)$ contained a nonzero point of
order $\ell$, normality and irreducibility would imply $E[\ell]\subset E(F)$.
Then $\GL_2(\F_\ell)$ would be a quotient of the abelian group
$\Gal(F/\Q)$, a contradiction. Thus $E(F)_{\tors}=0$.
\end{proof}

For $E/\Q$ satisfying Assumption~\ref{ass:working}, and for the class number one discriminant $d_0$, choose $\tau_0\in\CM_N(d_0)$, defined over $K_0=\Q(\sqrt{d_0})$, and define
\begin{equation}\label{eq:normalized}
 \varphi_E:X^N\longrightarrow E,\qquad
 \varphi_E(\tau)=\alpha([\tau]-[\tau_0]).
\end{equation}
For every $\tau\in\CM_N(d_0)$, Theorem~\ref{thm:untwisted-gz} gives
$\alpha([\tau]-\xi)=0$ in $E(K_0)\otimes\Q$. The same holds for $\tau_0$,
so $\varphi_E(\tau)$ is torsion in $E(K_0)$ and therefore zero.

The normalization $\varphi_E$, initially defined over $\Q(\sqrt{d_0})$, descends to $\Q$. Indeed, every Galois conjugate
$\tau_0^\sigma$ is again in $\CM_N(d_0)$, so this vanishing gives
$\alpha([\tau_0^\sigma]-[\tau_0])=O$ and hence
$\varphi_E^\sigma=\varphi_E$.

The definition~\eqref{eq:normalized} and this vanishing remove the translation term in
Atkin--Lehner equivariance: that term is $\varphi_E(W_m\tau_0)=O$.
Consequently,
\begin{equation}\label{eq:equivariance}
 \varphi_E\circ W_m=[\varepsilon_m]\circ\varphi_E,
 \qquad \varepsilon_m=\prod_{\ell\mid m}a_\ell
 \quad(m\mid N),
\end{equation}
where the quaternionic $W_\ell$-eigenvalue is $a_\ell$, the negative of the classical modular Atkin--Lehner sign~\cite[Theorem~1.2]{BertoliniDarmon1996}.

\begin{proposition}\label{prop:classone}
Under Assumption~\ref{ass:working}, one has $\varphi_E(\tau)=O$ for every CM point $\tau$ of discriminant $D$ with $h(D)=1$ on $X^{N_E}$.
\end{proposition}
\begin{proof}
If $a_\ell=1$ for every $\ell\mid\gcd(N,D)$, then Theorem~\ref{thm:untwisted-gz} implies that $\alpha([\tau]-\xi)=0$ in $E(H_D)\otimes\Q=E(K_D)\otimes\Q$. The same holds for $\tau_0$. Hence $\varphi_E(\tau)$ is torsion in $E(K_0K_D)$, and the preceding torsion Lemma \ref{lem:biquadratic-torsion-free} makes it zero. Otherwise, some prime $\ell\mid\gcd(N,D)$ has $a_\ell=-1$. Since $h(D)=1$, ramified reciprocity~\cite[Lemma~5.9]{GonzalezRotger2006} gives $W_\ell\tau=\tau$. Equation~\eqref{eq:equivariance} then gives $2\varphi_E(\tau)=O$, so again $\varphi_E(\tau)=O$.
\end{proof}

Therefore, one can solve the class number one problem by studying the fiber of $\varphi_E$ above the identity. Our strategy is to determine this single fiber. If it is entirely accounted for by known CM points, then no hypothetical class-number-one fundamental discriminant with $|D|>4q$ could exist. Then a finite amount of check would find all class-number-one fundamental discriminant with $|D|\leq4q$. We establish the fiber by exhibiting CM points whose total contribution, including ramification, equals the parametrization degree.

\section{Degrees and Atkin--Lehner fixed points}\label{sec:involutions}

\subsection{The Takahashi degree comparison}\label{subsec:degree}
The Ribet--Takahashi degree comparison~\cite{RibetTakahashi1997,Takahashi2001}, in the no-isogeny specialization stated in~\cite[\S6.2]{Pasten2024}, gives
\begin{equation}\label{eq:degree}
 \deg\varphi_E=\delta_E=\frac{\delta_{0,E}}{t_pt_q},\qquad t_\ell=v_\ell(\Delta_E).
\end{equation}
Here $\delta_{0,E}$ is the classical modular degree, $\Delta_E$ is the minimal discriminant, and $t_p=v_p(\Delta_E^{\min})$. By the definition of \(\varphi_E\), its induced pushforward on Jacobians is \(\varphi_{E,*}=\alpha\). The standard identity \(\varphi_{E,*}\varphi_E^*=[\deg\varphi_E]\) therefore gives \(\alpha\alpha^\vee=[\deg\varphi_E]\), identifying the degree in the Ribet–Takahashi formula with that of our normalization. The degrees $\delta_E$ in Table~\ref{tab:complete-fibers} are calculated by this formula, using the data of $\delta_{0,E}$ from LMFDB.

\subsection{Fixed points and their images}\label{subsec:fixed-points}
The fixed-point formula of Ogg~\cite[\S1, p.~283]{Ogg1983}, in the form given in~\cite[Lemma~6]{Tu2014}, says, for $1<m\mid N$,
\begin{equation}\label{eq:fixed}
 \Fix(W_m)=\begin{cases}
 \CM_N(-4)\cup\CM_N(-8),&m=2,\\
 \CM_N(-m)\cup\CM_N(-4m),&m\equiv3\pmod4,\\
 \CM_N(-4m),&\text{otherwise}.
 \end{cases}
\end{equation}

\begin{lemma}\label{lem:fixed-torsion}
If $\varepsilon_m=-1$ and $x\in\Fix(W_m)$, then $\varphi_E(x)\in E[2]$. In particular, if $x\in\CM_N(D)$ and $E(H_D)[2]=\{O\}$, then $\varphi_E(x)=O$.
\end{lemma}
\begin{proof}
Equation~\eqref{eq:equivariance} gives $\varphi_E(x)=\varphi_E(W_mx)=-\varphi_E(x)$. Moreover, $x\in X^N(H_D)$ and $\varphi_E$ is defined over $\Q$.
\end{proof}

\begin{lemma}\label{lem:ringtorsion}
Let $E/\Q$ be semistable and admit no rational isogeny of prime degree.
Let $H_D$ be the ring class field of an order in an imaginary quadratic
field $K$, and let $\Delta_E$ be the minimal discriminant of $E$. Then
\begin{equation}\label{eq:ringtorsion}
 E(H_D)_{\tors}=0\qquad\text{if }K\ne\Q(\sqrt{\Delta_E}).
\end{equation}
\end{lemma}
\begin{proof}
For every prime $\ell$, the representation on $E[\ell]$ is irreducible
and therefore has image $\GL_2(\F_\ell)$ by
\cite[Proposition~1.1]{RibetStein2001}.
If $E(H_D)$ contains a nonzero point of order $\ell$, normality of
$H_D/\Q$ and irreducibility imply $E[\ell]\subset E(H_D)$.
Thus $\GL_2(\F_\ell)$ would be a quotient of
$\Gal(H_D/\Q)=A\rtimes\langle\iota\rangle$, where
$A=\Gal(H_D/K)$ is abelian and $\iota$ acts by inversion.
This group is metabelian, whereas for odd $\ell$ the commutator subgroup
of $\GL_2(\F_\ell)$ is the nonabelian group $\operatorname{SL}_2(\F_\ell)$.
Consequently, $\ell=2$ if $E(H_D)_{\tors}$ is non-trivial.
In a surjection $A\rtimes\langle\iota\rangle\twoheadrightarrow S_3$,
the image of $A$ is necessarily $A_3$, so the quadratic subfield of
$\Q(E[2])$ is $K$. This quadratic subfield is
$\Q(\sqrt{\Delta_E})$.
\end{proof}

For all pairs $(E,D)$ occurring in~\eqref{eq:fixed} with $E$ in
Theorem~\ref{thm:fibers}, $\varepsilon_m=-1$, and $\CM_N(D)\ne\varnothing$,
one has $\Q(\sqrt D)\ne\Q(\sqrt{\Delta_E})$, except when
$(E,D)=(\labelE{1594.a1},-3188)$.
Lemma~\ref{lem:ringtorsion} therefore gives $E(H_D)_{\tors}=0$ in every
nonexceptional case.

\begin{corollary}\label{cor:fixed-zero}
For every curve in Theorem~\ref{thm:fibers} and every $W_m$ with $\varepsilon_m=-1$, one has $\Fix(W_m)\subseteq\varphi_E^{-1}(O)$.
\end{corollary}
\begin{proof}
Lemmas~\ref{lem:fixed-torsion} and~\ref{lem:ringtorsion} settle every CM set except the single orbit in the exceptional case. In the case $(E,D)=(\labelE{1594.a1},-3188)$, we use a parity argument. Denote $p=2$ and $q=797$.
The parametrization factors as
\[
  \varphi_E=\psi\circ\pi,\qquad
  \pi:X^N\longrightarrow Y:=X^N/\langle W_p\rangle,
  \qquad \deg\psi=\delta_E/2=48.
\]
Let $\overline{W}_q$ denote the involution induced by $W_q$ on $Y$.
Since $\varepsilon_q=-1$, we have
$\psi\circ\overline{W}_q=[-1]\circ\psi$.

The key observation is the parity relation
\[
  \#\bigl(\operatorname{Fix}(\overline{W}_q)
       \cap\psi^{-1}(O)\bigr)
  \equiv \deg\psi \equiv 0 \pmod 2.
\]
Indeed, points of $\psi^{-1}(O)$ not fixed by $\overline{W}_q$ occur in
$\overline{W}_q$-pairs. At a fixed point above $O$, choose local coordinates in which
$\overline{W}_q$ and $[-1]$ both act by negation. The equivariance
identity makes the local expansion of $\psi$ odd, so the ramification
index is odd. 

Now put
\[
  A=\operatorname{Fix}(W_q)=\operatorname{CM}_N(-3188),
  \qquad
  B=\operatorname{Fix}(W_{pq})=\operatorname{CM}_N(-6376).
\]
By \eqref{eq:fixed} and Table \ref{tab:complete-fibers}, $\#A=30$ and $\#B=34$, while Lemma~\ref{lem:fixed-torsion}
and ~\ref{lem:ringtorsion} give $\varphi_E(B)=\{O\}$.
The involution $W_p$ acts freely on $A\sqcup B$, and
\[
  \operatorname{Fix}(\overline{W}_q)
      =\pi(A)\sqcup\pi(B),
  \qquad
  \#\pi(A)=15,\qquad \#\pi(B)=17.
\]

By the reciprocity law, either all of \(A\) lies in the identity fiber or none does.
Consequently,
\[
  \operatorname{Fix}(\overline{W}_q)\cap\psi^{-1}(O)
  =
  \begin{cases}
    \pi(B),\\
    \pi(A)\sqcup\pi(B),
  \end{cases}
\]
with respective cardinalities $17$ and $32$.
The parity excludes the first possibility.
Thus $\varphi_E(A)=\{O\}$, completing the exceptional case.
\end{proof}

If $\varepsilon_m=1$, the map $\varphi_E$ factors through $X^N/\langle W_m\rangle$. Each $W_m$-fixed point already known to lie above $O$ therefore contributes at least twice to the fiber. In particular, the positive involution $W_{43}$ on $\labelE{817.a1}$ doubles its two CM$(-43)$ points. On $\labelE{1466.b1}$ and $\labelE{1594.a1}$, $W_2$ doubles the two CM$(-8)$ points. 

Counting ramification, this already explains the rows of Table~\ref{tab:complete-fibers} with no additional CM points.

\section{Additional CM points}\label{sec:characters}
As Table~\ref{tab:complete-fibers} shows, the fiber may contain CM points beyond the class-number-one points and fixed points; these additional points arise from ramified local conditions, extra vanishing of twisted $L$-functions, and norm-compatibility relations.  

For $x\in\CM_N(D)$ and a ring class character $\chi$, form the twisted CM trace
\[
 P_\chi(x)=\sum_{\sigma\in\Gal(H_D/K)}\chi(\sigma)^{-1}\varphi_E(x^\sigma)
 \quad\text{in }E(H_D)\otimes\C.
\]

If all the projections $P_\chi(x)$ vanish, Fourier inversion shows that
$\varphi_E(x)$ is torsion. Hence it is $O$ whenever
$K\ne\Q(\sqrt{\Delta_E})$, by Lemma~\ref{lem:ringtorsion}.

\subsection{Ramified vanishing}
When $\gcd(N,D)\neq 1$, $P_\chi(x)$ is forced to be zero in certain cases similar to the classical argument in~\cite[Lemma~11.1]{Gross1984}. 
\begin{lemma}
\label{lem:gross-ramified}
    Let $x\in\CM_N(D)$, where $D=dc^2$ and $(c,N)=1$, and let $\chi$ be
any character of $\Gal(H_D/K)$.
For $r\mid\gcd(N,d)$, let $\mathfrak r$ be the prime of $\OO_c$
above $r$. Then
\begin{equation}\label{eq:localkill}
 \chi([\mathfrak r])\ne a_r\quad\Longrightarrow\quad P_\chi(x)=0.
\end{equation}
\end{lemma}

\begin{proof}
    Since $r\nmid c$, the ideal $\mathfrak r$ is invertible and
$[\mathfrak r]^2=1$. The ramified reciprocity law~\cite[Lemma~5.9]{GonzalezRotger2006} identifies its action on CM points with $W_r$. Write $\sigma_{\mathfrak r}$ for its Artin symbol.
The equivariance of $\varphi_E$ and a change of index give
\[
 a_rP_\chi(x)
 =\sum_\sigma\chi(\sigma)^{-1}
      \varphi_E(W_rx^\sigma)
 =\sum_\sigma\chi(\sigma)^{-1}
      \varphi_E(x^{\sigma\sigma_{\mathfrak r}})
 =\chi([\mathfrak r])P_\chi(x).
\]
Thus, if $\chi([\mathfrak r])\ne a_r$, then $P_\chi(x)=0$.
\end{proof}
This explains the following points: on $\labelE{1466.b1}$,
the orders $D=-20,-24,-52,-72$, and on $\labelE{1594.a1}$,
the orders $D=-20,-72,-88,-148$, have class number two and nontrivial
character value $-1$ at the ramified class above two. Since $a_2=1$,
each contributes four identity-fiber points by~\eqref{eq:localkill} and
Lemma~\ref{lem:ringtorsion}.

\begingroup
\small\centering
\begin{tabular}{@{}llll@{}}\toprule
$E$ & $D$ & $r$ & $(a_r,\chi([\mathfrak r]))$\\\midrule
\labelE{1466.b1} & $-20,\ -24,\ -52,\ -72$ & $2$ & $(1,-1)$\\
\labelE{1594.a1} & $-20,\ -72,\ -88,\ -148$ & $2$ & $(1,-1)$\\\bottomrule
\end{tabular}
\par
\endgroup

\subsection{Extra vanishing of twisted characters}
\label{subsec:twisted-gz}
Some CM points are explained by the following twisted version of the Gross--Zagier formula.
\begin{theorem}[Twisted Gross--Zagier--Zhang]\label{thm:twisted-gz}
Suppose $D=df^2$ and $N$ is of nonsplit type
modulo $d$ with $\gcd(f,N)=1$. Suppose that
\begin{equation}\label{eq:twisted-gz-local-condition}
 \chi([\mathfrak l])=a_\ell
       \qquad(\ell\mid\gcd(N,d)).
\end{equation}
Then $L(E/K,\chi,s)$ has sign $-1$, and there exists $C_{E,D,\chi}\ge0$ such that
\begin{equation}\label{eq:twisted-gz}
 \widehat h_E\bigl(P_\chi(x)\bigr)
       =C_{E,D,\chi}L'(E/K,\chi,1).
\end{equation}
\end{theorem}
This version is the twisted elliptic specialization of~\cite[Theorem~1.2]{YuanZhangZhang2013}, in the form given in~\cite[\S2C, equation~(2-4)]{CaiShuTian2014}. Thus, the extra vanishing of $L'(E/K,\chi,1)$ implies $P_\chi(x)=0$. The most common case of vanishing occurs for quadratic characters, as in~\cite[\S6]{CaeiroDarmon2025}. For a nontrivial quadratic ring class character $\chi$, genus theory and Artin
formalism give~\cite[\S\S6--7]{Cox2013}
\cite[Chapter~VII]{Neukirch1999}
\begin{equation}\label{eq:genusfactor}
 L(E/K,\chi,s)=L(E^{d_1},s)L(E^{d_2},s).
\end{equation}
If one factor has extra vanishing at the center, then $L'(E/K,\chi,1)$ vanishes as well.

The following examples are explained by Theorem~\ref{thm:twisted-gz}. 
For the positive constituent $d_2$, one has $r_{\operatorname{an}}(E^{d_2})\ge2$, as verified below.
\begin{center}\small
\begin{tabular}{@{}lll@{}}\toprule
$E$ & $D$ & $(d_1,d_2)$\\\midrule
\labelE{1477.a1} & $-36,\ -148$ & $(-3,12),\ (-4,37)$\\
\labelE{1909.a1} & $-75,\ -100,\ -667$ & $(-15,5),\ (-20,5),\ (-23,29)$\\\bottomrule
\end{tabular}
\end{center}
The following table gives integral models $[a_1,a_2,a_3,a_4,a_6]$ of
these twists and non-torsion rational points on them; 
\begin{center}\small
\begin{tabular}{@{}llll@{}}\toprule
$E$ & $d_2$ & $[a_1,a_2,a_3,a_4,a_6]$ & $P$\\\midrule
\labelE{1477.a1} & $12$ & $[0,0,0,-867,12962]$ & $(-23,144)$\\
\labelE{1477.a1} & $37$ & $[1,9,0,-8214,354571]$ & $(26,393)$\\
\labelE{1909.a1} & $5$ & $[0,0,1,-100,281]$ & $(-1,19)$\\
\labelE{1909.a1} & $29$ & $[0,0,1,-3364,54875]$ & $(-13,310)$\\\bottomrule
\end{tabular}
\end{center}
The root-number formula gives
\[
 w(E^{d_2})=w(E)\left(\frac{d_2}{-N_E}\right)=1
\]
in all four cases. Gross--Zagier--Kolyvagin~\cite[Theorem~3.22]{Darmon2004}
therefore excludes analytic rank zero, and the positive root number gives
$r_{\operatorname{an}}(E^{d_2})\ge2$.

The first four orders have class number two and $\gcd(N_E,D)=1$. Their nontrivial characters
satisfy the condition of Theorem~\ref{thm:twisted-gz}, so $P_\chi(x)=0$. Together with the vanishing of the untwisted trace, this shows that these points lie in their corresponding fibers. 

For $D=-667$, one has $\gcd(N_E,D)=23$, and the class group is cyclic of order four, with $\Pic(\mathcal{O}_D)=\Z/4\Z$. The ramified class above $23$ is its element of order two. The quadratic character
takes value $1=a_{23}$, so Theorem~\ref{thm:twisted-gz} applies, with
constituents $(-23,29)$. Its quartic characters take value $-1$ there
and vanish by~\eqref{eq:localkill} as well. Therefore, all character components have
been accounted for, and we conclude that $\CM_{1909}(-667)$ lies in the fiber.

\subsection{Norm-compatibility relations}
Certain CM points of non-maximal orders may vanish due to norm-compatibility relations. Fix $K=\Q(\sqrt d)$. Write
\[
 \OO_f=\Z+f\OO_K,\qquad F_f=H_{df^2},\qquad
 \mathrm{HP}(f)=\varphi_E\bigl(\CM_N(df^2)\bigr)\subset E(F_f).
\]
One has the following norm-compatibility relations (see~\cite[\S6.2, Corollary~6.6]{CornutVatsal2007}).
\begin{proposition}[Norm compatibility]
\label{prop:norm-compatibility}
Assume $\OO_f^\times=\{\pm1\}$, and let
$\ell\nmid Nf$ be prime. For every $P_{f\ell}\in\mathrm{HP}(f\ell)$,
there exists $P_f\in\mathrm{HP}(f)$ such that
\begin{equation}\label{eq:normscalar}
 \operatorname{Tr}_{F_{f\ell}/F_f}(P_{f\ell})
 =
 \begin{cases}
 [a_\ell]P_f,
   & \ell \text{ inert in }K,\\
 [a_\ell]P_f-P_f^{\sigma_\lambda}
              -P_f^{\sigma_\lambda^{-1}},
   &  \ell=\lambda\bar\lambda
       \text{ split},\\
 [a_\ell]P_f-P_f^{\sigma_\lambda},
   &  \ell=\lambda^2
       \text{ ramified}.
 \end{cases}
\end{equation}
Here $\sigma_\lambda\in\Gal(F_f/K)$ is the Artin symbol of
$\lambda$.
\end{proposition}

If $F_{f\ell}=F_f$, the left-hand side of
\eqref{eq:normscalar} is the single point $P_{f\ell}$.
Consequently, vanishing of the right-hand side places the
corresponding CM point in $\varphi_E^{-1}(O)$.
The same conclusion holds if the right-hand side is torsion
and $E(F_f)_{\tors}=0$.

For $d=-15$, one has $F_2=F_1$, since
$h(-60)=h(-15)=2$. The prime above $2$ represents the
nontrivial ideal class, so its Artin symbol $\sigma$ generates
$\Gal(F_1/K)$ of order two. For any $P_2\in\mathrm{HP}(2)$,
choose $P_1$ as in Proposition~\ref{prop:norm-compatibility}.
The untwisted Gross--Zagier vanishing and
Lemma~\ref{lem:ringtorsion} give
\[
 P_1+P_1^\sigma=O.
\]
Since $2$ splits in $K$, equation~\eqref{eq:normscalar} becomes
\[
 P_2=[a_2]P_1-2P_1^\sigma=[a_2+2]P_1.
\]
On $\labelE{707.a1}$, one has $a_2=-2$. Hence every
CM point of discriminant $-60$ lies in the identity fiber,
giving all eight such points.

We have therefore accounted for every point listed in
Table~\ref{tab:complete-fibers} and completed the proof of
Theorem~\ref{thm:fibers}.

\section{A finiteness conjecture for CM fibers}\label{sec:further}
We conclude the paper by proposing the following conjecture. 

\begin{conjecture}
There are only finitely many elliptic curves $E/\Q$ satisfying
Assumption \ref{ass:working} whose identity fibers $\varphi_E^{-1}(O)$
consist entirely of CM points.
\end{conjecture}

This conjecture is supported by the following estimates: Let $E$ satisfy Assumption \ref{ass:working}. Watkins's estimate \cite[Lemma~2.1 and Theorem~5.1]{Watkins} and the Ribet--Takahashi comparison
\cite[\S6.2]{Pasten2024} give, for sufficiently large $N$,
\[
  \delta_E
  =\frac{\delta_{0,E}}{t_pt_q}
  \ge \frac{N|\Delta_E|^{1/6}}{5350\,t_pt_q\log N}
  \ge \frac{N^{7/6}}{85600\log N}.
\]
Here we use $|\Delta_E|=p^{t_p}q^{t_q}$ and
$\ell^{(t-1)/6}/t>1/4$ for a prime $\ell\ge2$ and $t\ge1$.

The familiar CM contributions grow much more slowly. Let $S_N$ be the union of the
class-number-one points and all Atkin--Lehner fixed points. The contribution of the class-number-one points is finite. Moreover, since $h(D)\ll\sqrt{|D|}\log|D|$, the contribution of all fixed points \eqref{eq:fixed} is of the size 

\[\sum_{\substack{m\mid N\\m>1}}\#\operatorname{Fix}(W_m)\ll \sqrt N\log N.\]

Ramification cannot compensate for this shortage asymptotically.
The genus formula gives $g(X^N)=O(N)$ \cite{Voight2021}, while
Riemann--Hurwitz 
\[
\sum_{x\in X^N}(e_x-1)=2g(X^N)-2,
\]
gives that the ramification is of size $2g-2=O(N)$.
Consequently, an entirely CM fiber
would require at least $\delta_E-O(N)=(1-o(1))\delta_E$ additional
CM points outside the $S_N$.

In practice, these additional points appear as comparatively small
contributions arising from special local relations or extra vanishing
of twisted $L$-functions. We expect them to remain too sparse to supply
almost the entire fiber as $N$ grows. This expectation, together with the
degree comparison above, motivates the above conjecture.

\section*{Acknowledgements}
The author would like to thank Michael Lipnowski for suggesting the research problem. The author appreciates Elias Caeiro and Henri Darmon for valuable discussions and comments on this work.

\section*{Disclosure of AI usage}
GPT-6 was used to revise and proofread the manuscript and to help discover Lemmas~\ref{lem:biquadratic-torsion-free} and~\ref{lem:ringtorsion} and the exceptional case of Corollary~\ref{cor:fixed-zero}, which help simplify the arguments but are not essential. The main idea of the paper and its main results, Theorem~\ref{thm:fibers} and Corollary~\ref{cor:classone}, are the author's own work.

\providecommand{\bysame}{\leavevmode\hbox to3em{\hrulefill}\thinspace}
\providecommand{\MR}{\relax\ifhmode\unskip\space\fi MR }
\providecommand{\MRhref}[2]{%
  \href{http://www.ams.org/mathscinet-getitem?mr=#1}{#2}
}
\providecommand{\href}[2]{#2}

\begin{center}
  \begin{tabular}{c}
    \textbf{Jingxuan Geng} \\
    Department of Mathematics, The Ohio State University \\
    \texttt{geng.210@osu.edu}
  \end{tabular}
\end{center}

\begin{thebibliography}{99}

\bibitem{Baker1966}
Alan Baker, \emph{Linear forms in the logarithms of algebraic numbers},
Mathematika \textbf{13} (1966), no.~2, 204--216.
\newblock \href{https://doi.org/10.1112/S0025579300003971}{doi:\nolinkurl{10.1112/S0025579300003971}}.

\bibitem{Baker1969}
\bysame, \emph{A remark on the class number of quadratic fields},
Bulletin of the London Mathematical Society \textbf{1} (1969), no.~1, 98--102.
\newblock \href{https://doi.org/10.1112/blms/1.1.98}{doi:\nolinkurl{10.1112/blms/1.1.98}}.

\bibitem{BDMTV2023}
Jennifer~S. Balakrishnan, Netan Dogra, J.~Steffen M{\"u}ller, Jan Tuitman,
and Jan Vonk, \emph{Quadratic {Chabauty} for modular curves: algorithms
and examples}, Compositio Mathematica \textbf{159} (2023), no.~6, 1111--1152.
\newblock \href{https://doi.org/10.1112/S0010437X23007170}{doi:\nolinkurl{10.1112/S0010437X23007170}}.

\bibitem{Banks2026}
William~D. Banks, \emph{On the {OpenAI} proof that {Landau--Siegel}
zeros do not exist}, preprint, 2026.
\newblock \href{https://arxiv.org/abs/2610.15483}{arXiv:\nolinkurl{2610.15483}}.

\bibitem{BertoliniDarmon1996}
Massimo Bertolini and Henri Darmon, \emph{Heegner points on {Mumford--Tate}
curves}, Inventiones Mathematicae \textbf{126} (1996), no.~3, 413--456.
\newblock \href{https://doi.org/10.1007/s002220050105}{doi:\nolinkurl{10.1007/s002220050105}}.

\bibitem{BuhlerGrossZagier1985}
Joe~P. Buhler, Benedict~H. Gross, and Don~B. Zagier,
\emph{On the conjecture of {Birch} and {Swinnerton-Dyer} for an elliptic
curve of rank~3}, Mathematics of Computation \textbf{44} (1985),
no.~170, 473--481.
\newblock \href{https://doi.org/10.1090/S0025-5718-1985-0777279-X}{doi:\nolinkurl{10.1090/S0025-5718-1985-0777279-X}}.

\bibitem{CaeiroDarmon2025}
Elias Caeiro and Henri Darmon, \emph{The {Heegner--Stark} theorem and
{Stark--Heegner} points}, Essential Number Theory \textbf{4} (2025),
no.~1, 67--99.
\newblock \href{https://doi.org/10.2140/ent.2025.4.67}{doi:\nolinkurl{10.2140/ent.2025.4.67}}.

\bibitem{CaeiroDarmonUnpublished}
Elias Caeiro and Henri Darmon, \emph{{Stark--Heegner} points and real
quadratic fields of {Richaud--Degert} type}, unpublished manuscript.

\bibitem{CaeiroGengUnpublished}
Elias Caeiro and Jingxuan Geng, \emph{{Gauss} class number problems via
{Gross--Zagier} and $p$-adic uniformization of {Shimura} curves},
unpublished manuscript.

\bibitem{CaiShuTian2014}
Li Cai, Jie Shu, and Ye Tian, \emph{Explicit {Gross--Zagier} and
{Waldspurger} formulae}, Algebra \& Number Theory \textbf{8} (2014),
no.~10, 2523--2572.
\newblock \href{https://doi.org/10.2140/ant.2014.8.2523}{doi:\nolinkurl{10.2140/ant.2014.8.2523}}.

\bibitem{CornutVatsal2007}
Christophe Cornut and Vinayak Vatsal, \emph{Nontriviality of
{Rankin--Selberg} $L$-functions and {CM} points},
in \emph{$L$-Functions and Galois Representations},
D.~Burns, K.~Buzzard, and J.~Nekov{\'a}{\v r} (eds.),
London Mathematical Society Lecture Note Series, vol.~320,
Cambridge University Press, Cambridge, 2007, pp.~121--186.
\newblock \href{https://doi.org/10.1017/CBO9780511721267.005}{doi:\nolinkurl{10.1017/CBO9780511721267.005}}.

\bibitem{Cox2013}
David~A. Cox, \emph{Primes of the form $x^2+ny^2$: {Fermat}, class field
theory, and complex multiplication}, second ed.,
John Wiley \& Sons, Hoboken, NJ, 2013.
\newblock \href{https://doi.org/10.1002/9781118400722}{doi:\nolinkurl{10.1002/9781118400722}}.

\bibitem{Darmon2004}
Henri Darmon, \emph{Rational points on modular elliptic curves},
CBMS Regional Conference Series in Mathematics, vol.~101,
American Mathematical Society, Providence, RI, 2004.

\bibitem{Goldfeld1976}
Dorian~M. Goldfeld, \emph{The class number of quadratic fields and the
conjectures of {Birch} and {Swinnerton-Dyer}},
Annali della Scuola Normale Superiore di Pisa, Classe di Scienze
\textbf{3} (1976), no.~4, 623--663.

\bibitem{GonzalezRotger2006}
Josep Gonz{\'a}lez and Victor Rotger,
\emph{Non-elliptic {Shimura} curves of genus one},
Journal of the Mathematical Society of Japan \textbf{58} (2006),
no.~4, 927--948.
\newblock \href{https://doi.org/10.2969/jmsj/1179759530}{doi:\nolinkurl{10.2969/jmsj/1179759530}}.

\bibitem{Gross1984}
Benedict~H. Gross, \emph{Heegner points on $X_0(N)$},
in \emph{Modular Forms} (Durham, 1983), R.~A. Rankin (ed.),
Ellis Horwood, Chichester, 1984, pp.~87--105.

\bibitem{GrossZagier1986}
Benedict~H. Gross and Don~B. Zagier,
\emph{Heegner points and derivatives of $L$-series},
Inventiones Mathematicae \textbf{84} (1986), no.~2, 225--320.
\newblock \href{https://doi.org/10.1007/BF01388809}{doi:\nolinkurl{10.1007/BF01388809}}.

\bibitem{Heegner1952}
Kurt Heegner, \emph{Diophantische Analysis und Modulfunktionen},
Mathematische Zeitschrift \textbf{56} (1952), no.~3, 227--253.
\newblock \href{https://doi.org/10.1007/BF01174749}{doi:\nolinkurl{10.1007/BF01174749}}.

\bibitem{Neukirch1999}
J{\"u}rgen Neukirch, \emph{Algebraic number theory},
Grundlehren der mathematischen Wissenschaften, vol.~322,
Springer-Verlag, Berlin, 1999, translated by Norbert Schappacher.
\newblock \href{https://doi.org/10.1007/978-3-662-03983-0}{doi:\nolinkurl{10.1007/978-3-662-03983-0}}.

\bibitem{Oesterle1985}
Joseph Oesterl{\'e}, \emph{Nombres de classes des corps quadratiques
imaginaires}, S{\'e}minaire Bourbaki, 1983/84,
Ast{\'e}risque, vol.~121--122,
Soci{\'e}t{\'e} Math{\'e}matique de France, 1985,
Expos{\'e} no.~631, pp.~309--323.

\bibitem{Ogg1983}
Andrew~P. Ogg, \emph{Real points on {Shimura} curves},
in \emph{Arithmetic and Geometry}, Vol.~I,
Michael Artin and John Tate (eds.),
Progress in Mathematics, vol.~35,
Birkh{\"a}user, Boston, 1983, pp.~277--307.
\newblock \href{https://doi.org/10.1007/978-1-4757-9284-3_12}{doi:\nolinkurl{10.1007/978-1-4757-9284-3_12}}.

\bibitem{OpenAI2026}
{OpenAI}, \emph{Uniform exclusion of {Landau--Siegel} zeros},
OpenAI Math Release preprint, 1 October 2026.
\newblock \url{https://github.com/openai/math/blob/main/preprints/Uniform-exclusion-of-Landau-Siegel-zeros-October-1-2026/paper.pdf}.

\bibitem{Pasten2024}
Hector Pasten, \emph{{Shimura} curves and the $abc$ conjecture},
Journal of Number Theory \textbf{254} (2024), 214--335,
with an appendix by Robert Lemke Oliver and Jesse Thorner.
\newblock \href{https://doi.org/10.1016/j.jnt.2023.07.002}{doi:\nolinkurl{10.1016/j.jnt.2023.07.002}}.

\bibitem{RibetStein2001}
Kenneth~A. Ribet and William~A. Stein,
\emph{Lectures on {Serre}'s conjectures},
in \emph{Arithmetic Algebraic Geometry} (Park City, UT, 1999),
IAS/Park City Mathematics Series, vol.~9,
American Mathematical Society, Providence, RI, 2001, pp.~143--232.

\bibitem{RibetTakahashi1997}
Kenneth~A. Ribet and Shuzo Takahashi,
\emph{Parametrizations of elliptic curves by {Shimura} curves and by
classical modular curves},
Proceedings of the National Academy of Sciences of the United States
of America \textbf{94} (1997), no.~21, 11110--11114.
\newblock \href{https://doi.org/10.1073/pnas.94.21.11110}{doi:\nolinkurl{10.1073/pnas.94.21.11110}}.

\bibitem{SchoofTzanakis2012}
Ren{\'e} Schoof and Nikos Tzanakis,
\emph{Integral points of a modular curve of level~11},
Acta Arithmetica \textbf{152} (2012), no.~1, 39--49.
\newblock \href{https://doi.org/10.4064/aa152-1-4}{doi:\nolinkurl{10.4064/aa152-1-4}}.

\bibitem{Serre1997}
Jean-Pierre Serre, \emph{Lectures on the {Mordell--Weil} theorem},
third ed., Aspects of Mathematics, vol.~E15,
Friedr. Vieweg \& Sohn, Braunschweig, 1997,
translated and edited by Martin Brown.
\newblock \href{https://doi.org/10.1007/978-3-663-10632-6}{doi:\nolinkurl{10.1007/978-3-663-10632-6}}.

\bibitem{Stark1967}
Harold~M. Stark, \emph{A complete determination of the complex quadratic
fields of class-number one}, Michigan Mathematical Journal
\textbf{14} (1967), no.~1, 1--27.
\newblock \href{https://doi.org/10.1307/mmj/1028999653}{doi:\nolinkurl{10.1307/mmj/1028999653}}.

\bibitem{Stark1969}
\bysame, \emph{On the ``gap'' in a theorem of {Heegner}},
Journal of Number Theory \textbf{1} (1969), no.~1, 16--27.
\newblock \href{https://doi.org/10.1016/0022-314X(69)90023-7}{doi:\nolinkurl{10.1016/0022-314X(69)90023-7}}.

\bibitem{Takahashi2001}
Shuzo Takahashi, \emph{Degrees of parametrizations of elliptic curves
by {Shimura} curves}, Journal of Number Theory \textbf{90} (2001),
no.~1, 74--88.
\newblock \href{https://doi.org/10.1006/jnth.2000.2614}{doi:\nolinkurl{10.1006/jnth.2000.2614}}.

\bibitem{LMFDB}
{The LMFDB Collaboration},
\emph{The $L$-functions and modular forms database}, 2026,
Elliptic curves over the rationals; accessed 2 October 2026.
\newblock \url{https://www.lmfdb.org/EllipticCurve/Q/}.

\bibitem{Tu2014}
Fang-Ting Tu, \emph{Schwarzian differential equations associated to
{Shimura} curves of genus zero}, Pacific Journal of Mathematics
\textbf{269} (2014), no.~2, 453--489.
\newblock \href{https://doi.org/10.2140/pjm.2014.269.453}{doi:\nolinkurl{10.2140/pjm.2014.269.453}}.

\bibitem{Voight2021}
John Voight, \emph{Quaternion algebras},
Graduate Texts in Mathematics, vol.~288, Springer, Cham, 2021.
\newblock \href{https://doi.org/10.1007/978-3-030-56694-4}{doi:\nolinkurl{10.1007/978-3-030-56694-4}}.

\bibitem{Watkins}
Mark Watkins,
\emph{Explicit lower bounds on the modular degree of an elliptic curve},
revised author version of the 2004 preprint.
\newblock \href{https://arxiv.org/abs/math/0408126}{arXiv:\nolinkurl{math/0408126}}.

\bibitem{YuanZhangZhang2013}
Xinyi Yuan, Shou-Wu Zhang, and Wei Zhang,
\emph{The {Gross--Zagier} formula on {Shimura} curves},
Annals of Mathematics Studies, vol.~184,
Princeton University Press, Princeton, NJ, 2013.
\newblock \href{https://doi.org/10.1515/9781400845644}{doi:\nolinkurl{10.1515/9781400845644}}.

\end{thebibliography}
\end{document}